\documentclass[11pt, oneside]{article}   	
\usepackage{geometry}                		
\usepackage[parfill]{parskip}    			
\usepackage{graphicx}				
\usepackage{amssymb}
\usepackage{mathtools}
\usepackage{enumerate}
\usepackage{tikz}

\usetikzlibrary{arrows}

\usepackage{indentfirst}
\usepackage[utf8]{inputenc}
\usepackage{amsthm}
\usepackage{amsmath}

\usepackage{algorithmic}
\usepackage{algorithm}
\usepackage{authblk} 
\title{Dense halves in balanced bipartitions of $K_4$-free graphs}
\author{Yue Xu}
\author{Xiao-Dong Zhang\thanks{Corresponding author: xiaodong@sjtu.edu.cn}}

\date{September 2024}

\affil{School of Mathematical Sciences, MOE-LSC, SHL-MAC, \\
Shanghai Jiao Tong University, 800 Dongchuan Road, Shanghai, 200240, P.R. China}

\begin{document}

\maketitle

\begin{abstract}
    Let $G$ be a simple graph on even vertices. Denote $D_{2,\infty}^b(G)$ by the minimal value of $\max\{e(A),e(A^c)\}$ over all 2-balanced partition $V(G)=A\cup A^c$ where $|A|=|A^c|$. In this paper we first provide a lower bound on $D_{2,\infty}^b(G)$ by examining the graph  $2n-$blow-up of $I_7\vee C_5$, where $I_7$ represents a 7-vertex independent set and $C_5$ is a 5-cycle, which disproves Conjecture 2.6 in \cite{balogh2023}. Furthermore, we present two upper bounds on $D_{2,\infty}^b(G)$ for a $K_4$-free graph and a join of independent set and triangle-free graph.
\end{abstract}

\newtheorem{Def}{Definition}[section]
\newtheorem{Th}[Def]{Theorem}
\newtheorem{Prop}[Def]{Proposition}
\newtheorem{Conj}[Def]{Conjecture}
\newtheorem{Fact}[Def]{Fact}
\newtheorem{lemma}[Def]{Lemma}
\newtheorem{remark}[Def]{Remark}
\newtheorem{Cor}[Def]{Corollary}

\section{Introduction}

Let $G$ be a simple graph with vertex set $V(G)$ and edge set $E(G)$. We denote the number of vertices by $|G| = |V(G)|$ and the number of edges by $e(G) = |E(G)|$. Given vertex-subset $A$, we denote $e(G[A])$ the number of edges in $G$ with both endpoints from $A$. If $A$ and $B$ are two disjoint vertex subsets, we denote $e(G[A],G[B])$ the number of edges in $G$ connecting vertices from $A$ to vertices from $B$. When $G$ is clear from context, we also write $e(A)$ and $e(A,B)$. A vertex 2-partition of a graph $G$ is a partition $V(G)=A\cup A^c$ of its vertex subset into 2 classes. And we denote $D_2(G):=\min_{A}(e(A)+e(A^c))$ where the minimum is taken over all $2$-partitions. By definition, $D_2(G)$ also represents the minimum number of edges which have to be removed to make $G$ bipartite. Furthermore, 2-partition is balanced if $|A|=|A^c|$. And we denote
$$
\begin{aligned}
    D_2^b(G):=\min\limits_{A}(e(A)+e(A^c)); D_{2,\infty}^b(G):=\min\limits_{A}\left(\max\{e(A),e(A^c)\}\right)
\end{aligned}
$$
where the minimum is taken over all $2$-balanced-partitions $V(G)=A\cup A^c$. For a graph $H$, we define a graph $G$ as $H$-free if $G$ does not contain $H$ as a subgraph.

Given two graphs $G_1$ and $G_2$, the join of $G_1$ and $G_2$, denoted by $G_1\vee G_2$, is defined as a graph whose vertex set is $V(G_1) \cup V(G_2)$, and whose edge set is $E(G_1) \cup E(G_2) \cup { (v_1, v_2) \mid v_1 \in V(G_1), v_2 \in V(G_2) }$. For a graph $G$ and an integer $n$, the $n$-blow-up of $G$, denoted by $G^n$, is defined as a graph obtained by replacing every vertex $v$ of $G$ with $n$ vertices, where a copy of $u$ is adjacent to a copy of $v$ in the blow-up graph if and only if $u$ is adjacent to $v$ in $G$.  

Erdős once offered a \$250 reward to the solution of the sparse half conjecture:
\begin{Conj}\cite{erdos1988}
Let $G$ be a triangle-free graph on $n$ vertices. If $n$ is even, then there exists a a balanced 2-partition $V(G)=A\cup A^c$ such that
$$\min\{e(A),e(A^c)\}\leq \frac{n^2}{50}.$$
\end{Conj}

The most recent progress by Razborov\cite{razborov2022} demonstrates that every triangle-free graph $n$ has a subset of vertices of size $n/2$ spanning at most $(27/1024)n^2$ edges. Another related conjecture of Erdős stated that:

\begin{Conj}\cite{erdos1997}
    Let $G$ be a triangle-free graph on $n$ vertices.
    $$
    D_2(G)\leq \frac{n^2}{25}.
    $$
\end{Conj}

The most recent progress by Balogh, Clemen and Lidický\cite{balogh2021} demonstrates that every triangle-free graph $G$ on $n$ vertices can be made bipartite by deleting at most $n^2/23.5$ edges. Several years later, they also raised and solved 2 related problems: for every triangle-free graph on $n$ vertices, where $n$ is sufficiently large, $D_2^b(G)\leq n^2/16$ and $D_{2,\infty}^b(G)\leq n^2/18$. In addition, \cite{bollobas1999},\cite{scott2005},\cite{lee2013} investigate the related 2-partition problem concerning a fixed number of graph edges and solve some of these problems.

For $K_4$-free graphs, significant work has been done on related problems.  Sudakov \cite{sudakov2007} proved that every $K_4$-free graph $G$ on $n$ vertices can be made bipartite by deleting at most $n^2/9$ edges. Recently, Reiher \cite{reiher2022}, building on the work of Liu and Ma \cite{liu2022}, proved that every $K_4$-free graph contains a set of size $n/2$ spanning at most $n^2/18$, thereby solving the conjecture proposed in \cite{erdos1994}. Additionally, Balogh, Clemen and Lidický\cite{balogh2023} recently inquired about the upper bounds of $D_2^b(G)$ and $D_{2,\infty}^b(G)$ for $K_4$-free graph, and proposed the following conjectures:

\begin{Conj}
    Let $G$ be a $K_4$-free graph on $n$ vertices. If $n$ is even, then
    $$D_2^b(G)\leq \frac{n^2}{9}.$$
\end{Conj}

\begin{Conj}
    Let $G$ be a $K_4$-free graph on $n$ vertices. If $n$ is even, then
    $$D_{2,\infty}^b(G)\leq \frac{n^2}{16}.$$
\end{Conj}

Reiher's work\cite{reiher2022} mentioned in the previous paragraph implies that Conjectures 1.3 and 1.4 holds for regular graphs. Furthermore, Balogh, Clemen and Lidický\cite{balogh2023} demonstrated that Conjectures 1.3 and 1.4 holds for 3-partite graph.
\begin{Th}
Let $G$ be a 3-partite graph on $n$ vertices. If $n$ is even,
        $$D_{2,\infty}^b(G)\leq \frac{n^2}{16}.$$
The equality holds when $G$ is a complete 3-partite graph with class sizes $n/2,n/4$ and $n/4$.
\end{Th}

The main result of this paper is to provide a family of counterexamples to Conjecture 1.4, as well as to establish an upper bound on $D_{2,\infty}^b(G)$ for $K_4$-free graphs. Additionally, we present an upper bound on $D_{2,\infty}^b(G)$ for the join of an independent set and a triangle-free graph.

\begin{Th}
Let $H=I_7\vee C_5$, where $I_7$ is a 7-vertex independent set and $C_5$ is a 5-cycle.
    \begin{equation}
        \begin{aligned}
            D_{2,\infty}^b(H^{2n})=\frac{37}{24^2}|H^{2n}|^2>\frac{|H^{2n}|^2}{16}.
        \end{aligned}
    \end{equation} 
    where $H^{2n}$ denotes the $2n$-blow-up of $H$.
\end{Th} 

\begin{remark}
    Theorem 1.6 disproves Conjecture 1.4.
\end{remark}
\begin{Th}
Let $G$ be a $K_4$-free graph on $n$ vertices. If $n$ is even and greater than $10^5$,
    \begin{equation}
        D_{2,\infty}^b(G)< 0.074n^2.
    \end{equation}
\end{Th}

\begin{Th}
Let $G$ be the join of an independent set and a triangle-free set, where $|G|=n$. If $n$ is even and large enough,
    \begin{equation}
    D_{2,\infty}^b(G)\leq \frac{5}{72}n^2.
    \end{equation}
\end{Th}

The organization of this paper is as follows. In section 2, we prove Theorem 1.6. In section 3, we present some preliminaries for Theorem 1.8 and Theorem 1.9. In section 4 we complete the proof of Theorem 1.8, by considering two cases regrading the independent number. In section 5, we prove Theorem 1.9.

\section{The lower bound for the maximum value of $D_{2,\infty}^b(G)$}    

In this section, we will prove Theorem 1.6 and provide a lower bound on the maximum value of $D_{2,\infty}^b(G)$ for $K_4$-free $G$. 

\begin{Prop}
Let $H=I_7\vee C_5$, where $I_7$ is a 7-vertex independent set and $C_5$ is a 5-cycle.
$$D_{2,\infty}^b(H)= \frac{5}{72}|H|^2.$$
    
\end{Prop}
\begin{figure}
    \centering
    \includegraphics[width=0.5\linewidth]{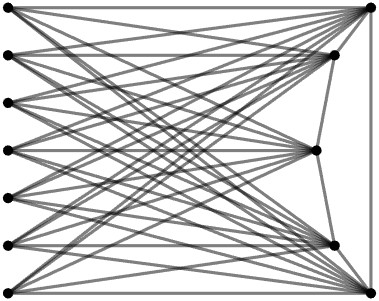}
    \caption{$H:I\vee C_5$}
    \label{fig:enter-label}
\end{figure}
\begin{proof}
    Let $V(H)=A\cup A^c$ be the 2-balanced partition of $H$. Without loss of generality, let $A$ be the vertex set that contains at least half of the vertices in $C_5$. According to the table below, $e(A)\geq 10=5/72\cdot |H|^2$. 
    \begin{table}[H]
        \centering
        \begin{tabular}{|c|c|c|c|}
        \hline
        vertices from $C_5$& min edges from original $C_5$   & edges from original $e(K,C_5)$   & $\min e(A)$   \\ \hline
        3& 1	& $3\times3=9$	& 10\\ \hline
        4& 3	& $4\times2=8$	& 11\\ \hline
        5& 5	& $5\times3=5$	& 10 \\ \hline
        \end{tabular}
        \label{tab:my_label}
    \end{table}
    
When $A$ contains $V(C_5)$, $\max\{e(A),e(A^c)\}=\max(10,0)=10$. Hence, $D_{2,\infty}^b(H)=10=5|H|^2/72$. 
\end{proof}

We can apply the similar method to determine the value of $D_{2,\infty}^b(H^{2n})$, where $H^{2n}$ represents the $2n$-blow-up of $H$. 

\begin{lemma}
Let $C_5^n$ be a graph with vertex set $V(C_5^n)=A_1\cup A_2\cup A_3\cup A_4\cup A_5$ and edges $xy\in C_5^n$ iff $x\in A_i$ and $y\in A_{i+1}$ for some $i\in[5]$, where $A_6:=A_1$ and $|A_i|=n$. Let $G_0$ be a subgraph of $C_5^n$ with $|V(G_0)\cap A_i|=a_i$ for $i\in[5]$ and we denote the values $b_1, b_2, b_3, b_4, b_5$ as the non-increasing rearrangement of $a_1, a_2, a_3, a_4, a_5$.

1. If $T(b)=b_1b_4+b_4b_3+b_3b_2+b_2b_5+b_5b_1$, then $e(G_0)=a_1a_2+a_2a_3+a_3a_4+a_4a_5+a_5a_1\geq T(b)$.

2. If $|V(G_0)|=pn+q$, where $p$ and $q$ are non-negative integers with $p<5$ and $q<n$, then $e(G)\geq f(n,p,q)=\begin{cases}
    0 & p\leq 1\\ qn & p=2 \\ n^2+2qn & p=3\\ 3n^2+2qn & p=4
\end{cases}$. Equality holds if
$b_i=\begin{cases} n & i\leq p\\ q & i=p+1 \\ 0 & \text{other cases}\end{cases}$.
    
\end{lemma}

\begin{figure}[htp]
    \centering
\begin{tikzpicture}[x=0.5pt,y=0.5pt,yscale=-1,xscale=1]

\draw   (192.04,111.21) .. controls (192.04,83.45) and (220.49,60.94) .. (255.59,60.94) .. controls (290.68,60.94) and (319.13,83.45) .. (319.13,111.21) .. controls (319.13,138.98) and (290.68,161.48) .. (255.59,161.48) .. controls (220.49,161.48) and (192.04,138.98) .. (192.04,111.21) -- cycle ;
\draw   (97,175.62) .. controls (97,146.99) and (123.38,123.78) .. (155.93,123.78) .. controls (188.47,123.78) and (214.85,146.99) .. (214.85,175.62) .. controls (214.85,204.25) and (188.47,227.46) .. (155.93,227.46) .. controls (123.38,227.46) and (97,204.25) .. (97,175.62) -- cycle ;
\draw   (293.88,177.28) .. controls (293.88,148.6) and (323.6,125.35) .. (360.27,125.35) .. controls (396.94,125.35) and (426.67,148.6) .. (426.67,177.28) .. controls (426.67,205.97) and (396.94,229.22) .. (360.27,229.22) .. controls (323.6,229.22) and (293.88,205.97) .. (293.88,177.28) -- cycle ;
\draw   (106.5,285.75) .. controls (106.5,257.2) and (134.22,234.06) .. (168.42,234.06) .. controls (202.61,234.06) and (230.33,257.2) .. (230.33,285.75) .. controls (230.33,314.29) and (202.61,337.43) .. (168.42,337.43) .. controls (134.22,337.43) and (106.5,314.29) .. (106.5,285.75) -- cycle ;
\draw   (252.74,286.2) .. controls (252.74,257.03) and (281.82,233.39) .. (317.71,233.39) .. controls (353.59,233.39) and (382.67,257.03) .. (382.67,286.2) .. controls (382.67,315.36) and (353.59,339) .. (317.71,339) .. controls (281.82,339) and (252.74,315.36) .. (252.74,286.2) -- cycle ;
\draw   (252.33,93.93) -- (376.97,172.48) -- (341.94,309.15) -- (180.64,307.58) -- (135.83,174.84) -- cycle ;
\draw   (177.38,213.54) .. controls (177.38,173.41) and (211.12,140.88) .. (252.74,140.88) .. controls (294.35,140.88) and (328.09,173.41) .. (328.09,213.54) .. controls (328.09,253.67) and (294.35,286.2) .. (252.74,286.2) .. controls (211.12,286.2) and (177.38,253.67) .. (177.38,213.54) -- cycle ;

\draw (255.18,61.37) node [anchor=north west][inner sep=0.75pt]   [align=left] {$\displaystyle A_{1}$};
\draw (359.86,125.78) node [anchor=north west][inner sep=0.75pt]   [align=left] {$\displaystyle A_{2}$};
\draw (317.3,233.82) node [anchor=north west][inner sep=0.75pt]   [align=left] {$\displaystyle A_{3}$};
\draw (160.01,234.49) node [anchor=north west][inner sep=0.75pt]   [align=left] {$\displaystyle A_{4}$};
\draw (155.52,124.21) node [anchor=north west][inner sep=0.75pt]   [align=left] {$\displaystyle A_{5}$};
\draw (240.62,198.04) node [anchor=north west][inner sep=0.75pt]   [align=left] {$\displaystyle G_0$};
\draw (245.32,143.84) node [anchor=north west][inner sep=0.75pt]   [align=left] {$\displaystyle b_{1}$};
\draw (303.98,182.33) node [anchor=north west][inner sep=0.75pt]   [align=left] {$\displaystyle b_{4}$};
\draw (188.3,180.76) node [anchor=north west][inner sep=0.75pt]   [align=left] {$\displaystyle b_{5}$};
\draw (195.63,248.31) node [anchor=north west][inner sep=0.75pt]   [align=left] {$\displaystyle b_{2}$};
\draw (281.98,247.53) node [anchor=north west][inner sep=0.75pt]   [align=left] {$\displaystyle b_{3}$};

\end{tikzpicture}
\caption{$C_5^{n}$ and its subgraph $G_0$}
    \label{fig:enter-label}
\end{figure}
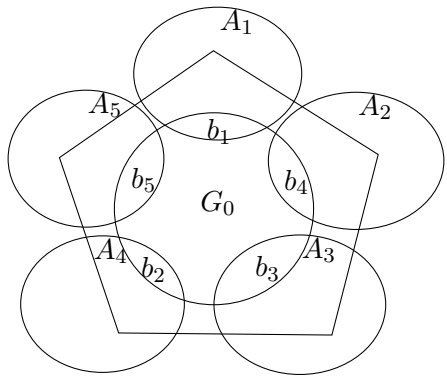

\begin{proof}
Without loss of generality, let $ a_1 $ denote the largest element among $\{a_n\}$, and assume that $ a_2 \geq a_5 $. Based on the arrangements of these elements in accordance with their magnitude, we obtain a total of $\frac{4!}{2} = 12$ distinct permutations. Let $ S_a = a_1 a_2 + a_1 a_5 + a_2 a_3 + a_3 a_4 + a_4 a_5 $ and $ T(b) = b_1 b_4 + b_1 b_5 + b_2 b_3 + b_2 b_5 + b_3 b_4 $ where $b=(b_1,b_2,b_3,b_4,b_5)$. 

Note that $(x_1y_1+x_2y_2)-(x_1y_2+x_2y_1)=(x_1-x_2)(y_1-y_2)$. Therefore, the following claim holds.

\textit{Claim 2.1} If $x_1\geq x_2$ and $y_1\geq y_2$, then $x_1y_1+x_2y_2\geq x_1y_2+x_2y_1$.

Next, we will show that $e(G_0)=S_a\geq T(b)$ always holds.

\textit{Case 1}: $b_2=a_2$, $b_3=a_5$. By Claim 2.1, $S_a\geq b_1b_2+b_1b_3+b_2b_5+b_4b_5+b_3b_4=T(b)+(b_3-b_5)(b_1-b_2)+(b_1-b_5)(b_2-b_4)\geq T(b)$. 
        
\textit{Case 2}: $b_2=a_2$, $b_4=a_5$. By Claim 2.1, $S_a\geq b_1b_2+b_1b_4+b_2b_5+b_3b_5+b_3b_4=T(b)+(b_2-b_5)(b_1-b_3)\geq T(b)$. 

\textit{Case 3}: $b_2=a_2$, $b_5=a_5$. By Claim 2.1, $S_a\geq b_1b_2+b_1b_5+b_2b_4+b_3b_4+b_3b_5=T(b)+(b_1-b_2)(b_3-b_4)+(b_1-b_5)(b_2-b_3)\geq T(b)$.

\textit{Case 4}: $b_3=a_2$, $b_4=a_5$. By Claim 2.1, $S_a\geq b_1b_3+b_1b_4+b_3b_5+b_2b_5+b_2b_4=T(b)+(b_1-b_3)(b_3-b_5)+(b_4-b_5)(b_2-b_3)\geq T(b)$.

\textit{Case 5}: $b_3=a_2$, $b_5=a_5$. By Claim 2.1, $S_a\geq b_1b_3+b_1b_5+b_3b_4+b_2b_4+b_2b_5=T(b)+(b_1-b_2)(b_3-b_4)\geq T(b)$. 

\textit{Case 6}: $b_4=a_2$,$b_5=a_5$.  By Claim 2.1, $S_a\geq T(b)$. Equality holds if $(b_2,b_3)=(a_3,a_2)$. 

Case 1 to 6 together show that $S_a\geq T(b)$ under different permutations.

(2)

If $i<j$, let $b^{ij}=(b_1^{i j},b_2^{i j},b_3^{i j},b_4^{i j},b_5^{i j})$, 
where $b_k^{i j}=\begin{cases}    b_k+1 & k=i \\    b_k-1 & k=j \\
    b_k   & \text{other cases}\end{cases}$.
Since $b_i\geq b_j$ for $i<j$, $T({b^{i j}})-T(b)\leq 0$ for $i<j$, according to the table given below.

\begin{table}[htb]
\begin{tabular}{|l|l|l|l|l|}
\hline
$T(b^{i j})-T(b)$ & $i=1$                & $i=2$                & $i=3$                & $i=4$      \\ \hline
$j=2$  & $b_4-b_3$           &                  &                      &            \\ \hline
$j=3$   & $b_5-b_2$           & $b_5-b_4+b_3-b_2-1$ &                  &            \\ \hline
$j=4$   & $b_5-b_3+b_4-b_1-1$ & $b_5-b_1$           & $b_2-b_1+b_4-b_3-1$ &         \\ \hline
$j=5$   & $b_5-b_2+b_4-b_1-1$ & $b_5-b_1+b_3-b_2-1$ & $b_4-b_1$           & $b_3-b_2$ \\ \hline
\end{tabular}
\caption{$T(b^{i j})-T(b)$}
\end{table}

Let $b^*=(b_1^{*},b_2^{*},b_3^{*},b_4^{*},b_5^{*})$, where $b_i^*=\begin{cases} n & i\leq p\\ q & i=p+1 \\ 0 & \text{other cases}\end{cases}$. For any $b$ satisfying $b_i\geq b_j$ for $i>j$ and $\sum\limits_{i=1}^5 b_i=pn+q$, we can apply the following algorithm to get a new value $\min(T(b))=T(b^*)$ not less than $T(b)$.
\begin{algorithm}[H]
    \caption{Algorithm of $\min(T(b))$}
    \label{alg:AOS}
    \renewcommand{\algorithmicrequire}{\textbf{Input:}}
    \renewcommand{\algorithmicensure}{\textbf{Output:}}
    \begin{algorithmic}[1]
        \REQUIRE b, b*   
        \ENSURE $\min(T(b))$    
        
        \STATE declare T as function defined in Lemma 2.2
        
        \WHILE {$b\neq b^*$}
            \STATE $x = \min(i)$ such that $b_i< n$, $y = \max(j)$ such that $b_j>0$
            \IF {$x<y$}
                \STATE $b_x=b_x+1$; $b_y=b_y-1$; $\min(T(b))=T(b)$;
            \ENDIF
        \ENDWHILE
        
        \RETURN $\min(T(b))$
    \end{algorithmic}
\end{algorithm}

Note that each operation in the while loop guarantees $b_i \geq b_j$ for $i > j$ and $\sum_{i=1}^5 b_i = p n + q $. In this case, according to Table 1, $T(b)$ does not increase after each iteration. In addition, $|| b - b^* ||$ decreases by 2 in each iteration, where $|| \cdot ||$ represents the sum of the absolute values of each component of the vector. This implies that after $|| b - b^* ||/2$  iterations, the while loop will terminate. At this time, $\min(T(b))=T(b^*)$.

Therefore, $e(G_0)\geq T(b^*)=f(n,p,q)$. Equality holds if $a_1\geq a_4\geq a_3\geq a_2\geq a_5$, and $\{a_n\}$ contains $p$ elements equal to $n$, 1 element equal to $q$, and the rest are equal to 0.
\end{proof}

\begin{remark}
    Actually, if $e(G_0)=T(b)=f(n,p,q)$, $T(b)$ should not decrease after each iteration. This implies if $b\neq b^*$, $b$ can only be $(b_1,b_2,0,0,0)$ and $(n,n,n,b_4,b_5)$ according to Table 1. Hence, if $pn+q=2n+q$, the equality holds if and only if $b=(n,n,q,0,0)$.
\end{remark}

Next, we will show $D_{2,\infty}^b(H^{2n})=37 n^2$ to prove Theorem 1.6.

\begin{proof}[Proof of Theorem 1.6]
   
    Let $V(H^{2n})=A\cup A^c$ be the 2-balanced partition of $H^{2n}$ such that $\max\{e(A), e(A^c)\} =D_{2,\infty}^b(H^{2n})$. Here, $|A|=12 n\geq |C_5^{2n}|$.  
    If $A\supset V(C_5^{2n})$, $e(A)=40n^2> 37n^2$.  
    Otherwise, let $|A\cap V(C_5^{2n})|=p\cdot 2n+ q$, where $p$ and $q$ are non-negative integers  with $p<5$ and $q<2n$. Without loss of generality, let $|A\cap V(C_5^{2n})|\geq |V(C_5^{2n})|/2 =5n$. If this condition does not hold, we can let $A = A^c$, where $V(H^{2n}) = A^c \cup (A^c)^c$ forms the same 2-balanced partition. Hence, $p\geq 2$, and when $p=2$, $q>n$.
    Therefore, by Lemma 2.7
    $$
    \begin{aligned}
        e(A)
        & = (|A|-|A\cap V(C_5^{2n})|)|A\cap V(C_5^{2n})|+e(A\cap V(C_5^{2n})) \\
        &\geq (2pn+q)(12n-2pn-q)+f(2n,p,q) \\
        &=\begin{cases}
        32n^2+q(6n-q) & p=2, n\leq q <2n  \\
        40n^2+4nq-q^2 & p=3, 0\leq q <2n  \\
        44n^2-q^2 & p=4, 0\leq q <2n
        \end{cases}.
    \end{aligned}
    $$
        
Hence, $e(A)\geq 37n^2$. Equality holds if and only if $|A\cap V(C_5^{2n})|$ and $e(A)=(2pn+q)(12n-2pn-q)+f(n,p,q)$. By Remark 2.3, this implies
$A$ contains all vertices in $A_1$, $A_4$ and half vertices in $A_3$ in Figure 2, then $e(A^c)=(|A^c|-|A^c\cap V(C_5^{2n})|)|A\cap V(C_5^{2n})|+e(A^c\cap V(C_5^{2n}))=5n\cdot7n+2n^2=37n^2$.

Therefore, $D_{2,\infty}^b(G)\geq e(A)\geq 37n^2$. Additionally, $D_{2,\infty}^b(G)=37n^2$ if and only if $A$ contains $7n$ vertices from $|G[A]\cap C_5^{2n}|=K_{n,2n}\cup I_{2n}$, where $K_{n,2n}$ represents a complete bipartite with class sizes $2n$ and $n$,and $I_{2n}$ represents a $2n$-vertex independents.
   
\end{proof}

\section{Some technical lemmas}

In this section, we will introduce the notations and lemmas that will be used later in this paper. Let $\alpha(G)$ be the independent number of $G$. For $v\in V(G)$, let $N_G(v)$ be the set of neighbors of $v$ in $G$ and $d_G(v)=|N_G(v)|$. If $G$ is clear context, we also write $N(v)$ and $d(v)$. We also omit floors and ceiling when they are not essential in our proofs.

The following 3 lemmas will be applied in this paper to estimate the upper bound. For completeness we include their proofs here.

\begin{lemma}
    Let $G$ be a $K_4$-free graph on $n$ vertices.\\
    1. If Z is a maximal triangle-free induced graph of G, then 
    \begin{equation}
        e(G)\leq \min\left\{\frac{|Z||G|}{2},|Z|(|G|-\frac 3 4|Z|)\right\}.
    \end{equation}    
    2. If $Z_0$ is a triangle-free induced graph of $G$ satisfying $|Z_0|\geq \frac{2}{3}|G|$, 
    \begin{align}
        e(G)\leq |Z_0|(|G|-\frac 3 4|Z_0|).
    \end{align}  
\end{lemma}
\begin{proof}
    For any vertex $v\in G$, where $G$ is $K_4$-free, its neighbor $N(v)$ is triangle-free, so $d(v)\leq |Z|$, which implies $e(G)=\frac{1}{2}\sum\limits_{v\in V(G)}d(v)\leq \frac{|Z||G|}{2}$. In addition, by Turán's Theorem $e(Z)\leq \frac{|Z|^2} 4$. Hence, $e(G)\leq e(Z)+e(Z,G\setminus Z)+2 e(G\setminus Z)\leq \frac{|Z|^2}4+\sum\limits_{v\in V(G\setminus Z)}d(v)\leq |Z|(|G|-\frac 3 4|Z|)$. 
    
    Clearly, the function $f(x):=x(|G|-\frac{3}{4}x)$ decreases with respect to $x\geq \frac{2}{3}|G|$. Hence, $|Z_0|\geq \frac 2 3|G|$ implies inequality (5).
\end{proof}

\begin{lemma}
    Let $G$ be a triangle-free graph on $n$ vertices.If $I_0$ is an independent set of $G$ with $|I_0|\geq \frac{|G|}{2}$, then  $e(G)\leq |I_0|(|G|-|I_0|)$.
\end{lemma}
\begin{proof}
    Let $I$ be a maximal independent set of triangle-free graph $G$. For any vertex $v\in G$, its neighbor $N(v)$ is an independent set, so $d(v)\leq |I|$, which implies $e(G)\leq e(I)+2e(I,G\setminus I)=\sum_{v\in V(G \setminus I)}d(v)\leq (|G|-|I|)|I|$. 
    
    Clearly, the function $f(x):=x(|G|-x)$ is decreasing with respect to $x\geq \frac n 2$. Hence, $|I_0|\geq  |G|/2 $ implies $e(G)\leq |I_0|(|G|-|I_0|)$.
\end{proof}

Moreover, if a $K_4$-free graph $G$ contains a large independent set, we can provide an upper bound by Lemma 3.1.

\begin{lemma}
    Let $G$ be a $K_4$-free graph on $n$ vertices and $I$ be an independent set of $G$. If $Z$ is a largest triangle-free induced graph of $G\setminus I$, then 
    \begin{align}
        e(G)\leq \min\left\{\frac{|Z|(|G|+|I|)}{2},|Z|(|G|-\frac 3 4|Z|)\right\}.
    \end{align}
    Moreover, if $|I|>\frac{|G|}{3}$, then
    \begin{align}
        e(G)\leq \frac{1}{4}(|G|-|I|)(|G|+3|I|)
    \end{align}
\end{lemma}

\begin{proof}
    For any vertex $v\in I$, its neighbor $N(v)$ is triangle-free and $N(v)\subset G\setminus I$. Hence, $d(v)=|N(v)|\leq |Z|$, which implies $e(V(I),V(G\setminus I))\leq \sum\limits_{v\in V(I)}d(v) \leq |Z||I|$. By Lemma 2.1, $e(G\setminus I)\leq \min\left\{\frac{|Z|(|G|-|I|)}{2},|Z|(|G|-|I|-\frac{3}{4} |Z|\right\}$. Therefore, $e(G)=e(I)+e(G\setminus I)+e(G\setminus I)\leq \min\left\{\frac{|Z|(|G|+|I|)}{2},|Z|(|G|-\frac 3 4|Z|)\right\}$.

    Clearly, the function $f(x):=x(|G|-\frac 3 4 x)$ is decreasing with respect to $x\geq \frac{2}{3}|G|$. Hence, $|Z|\leq|G\setminus I|<\frac 2 3|G|$ implies $e(G)\leq  |Z|(|G|-\frac 3 4|Z|)\leq (|G|-|I|)(\frac 1 4 |G|+\frac 3 4(I))$.
\end{proof}

The following conclusions can help us find a proper balanced 2-partition.
\begin{Prop}
    Let $G$ be a graph on $n$ vertices and $V(G)=B\cup C$. If $ p\in\{0,1,\ldots,|V(B)|\}$, there exists a vertex subset $Y$ (where $Y\subset B$) of size $p$ such that $$e(Y,C)\leq \frac{p}{|B|}e(B,C)$$. 
\end{Prop}
This proposition can be found in the literature (e.g.,\cite{krivelevich1995}).

\begin{Th}\cite{xu2010}
    Let $G$ be a graph with $n$ vertices and $m$ edges. If $n$ is even,
then $G$ admits a balanced 2-partition $V(G)=A\cup A^c$ such that
$$
\max\{e(A), e(A^c)\} \leq \frac{m  + \Delta(G) - \delta(G)}{4}
$$
where $\Delta(G)$ and $\delta(G)$ denote the maximum and minimum degree of the graph G.
\end{Th}

\begin{Cor}
     Let $G$ be an $n$-vertex graph where n is an even number. If $|E(G)|\leq (4m_0-\epsilon)n^2$ with $\epsilon>0$ and $n>1/\epsilon$, then there exists a balanced $2$-partition $A\cup A^c$ such that $\max\{e(A),e(A^c)\}<m_0 n^2$.
\end{Cor}
 In this paper, we set $m_0=0.074, \epsilon=10^{-5}$ in Section 4, and $m_0=\frac{5}{72}, \epsilon=0.001$ in Section 5. 

We also need the following theorems to further deal with the triangle-free graphs.
\begin{Th}\cite{balogh2021}
     Let $H$ be a triangle-free graph on $n$ vertices. If $|E(H)| \geq 0.3197\binom{n}{2}$ with $n$ large enough, $$D_2(H) \leq \frac{n^2}{25}.$$  
\end{Th}

\begin{Th}\cite{erdos1988}
     For a $K_3$-free graph with $n$ vertices and $m$ edges,
    $D_2(G) \leq  m-\frac{4 m^2}{n^2}$.
\end{Th}

\section{A new upper bound for $D_{2,\infty}^b(G)$}

In this section we assume that the order of $G$ is even and large enough. We will prove Theorem 1.8 by considering two cases based on the independent number of $G$.

\begin{lemma}
Let $G$ be a $K_4$-free graph on $n$ vertices, where $n>10^5$. If $G$ contains an independent set $I$ of size $0.28n$, then
    $$
        D_{2,\infty}^b(G)< 0.074n^2.
    $$
\end{lemma}  
\begin{proof}
    Let $Z$ be a maximal triangle-free induced subgraph in $G\setminus I$ and $|Z|=z n$.

\textbf{Case 1.1}: $|Z|\leq 0.45n$. By (6) of Lemma 3.3, 
    $$
    \begin{aligned}
        e(G) &\leq \frac{|Z|(|G|+|I|)}{2} \\
             &= (0.45n\cdot 0.64n)= 0.288n^2
    \end{aligned}
    $$     
Therefore, by Corollary 3.6, $D_{2,\infty}^b(G)< 0.074 n^2$. 

\textbf{Case 1.2}: $|Z|> 0.45n$. Now, there exists a triangle-free induced graph $Z_0\in G\setminus I$ of size $0.45n$. Choose a vertex subset $A$ of size $n/2$ such that $A\supset I$ and $A^c\supset Z_0$.    
    
    Since $0.28n>\frac{|A|}{3}$, by (7) of Lemma 3.3, $e(A)\leq \frac{1}{4}(|A|-|I|)(|A|+3|I|)=0.0737n^2<0.074n^2$. On the other hand, since the triangle-free graph $Z_0\subset A^c$ and $|Z_0|\geq \frac{2}{3}|A^c|$, by ,
    $$
    e(A^c)\leq |Z_0|(|A^c|-\frac{3}{4}|Z_0|)=(0.45n\cdot (0.45n-\frac 3 4\cdot 0.45n))=0.073125 n^2.
    $$
Therefore, $\max\{e(A),e(A^c)\}<0.074n^2$.
    
\end{proof}

According to Corollary 3.6, we only need to consider the graph whose edges are more than $0.02959n^2$. Hence, we introduce the following lemma.
\begin{lemma}
     Let $G$ be a graph with $n$ vertices and $m$ edges. There exists an edge $v_1v_2$ such that 
     \begin{equation}
        d(v_1)+d(v_2)\geq \frac{4m}{n}.
     \end{equation}
    If $G$ is $K_4$-free and $m\geq \frac{n^2} 4$, the independent number of $G$ satisfies $\alpha(G)\geq \frac{4m}n -n.$ 
\end{lemma}

\begin{proof}[Proof]
    By Cauchy-Schwarz inequality,
    $$
    \begin{aligned}
        \sum\limits_{ab\text{ is an edge}}(d(a)+d(b))
        =\sum\limits_{v\in V(G)}d(v)^2\geq \frac{\left(\sum\limits_{v\in V(G)}d(v)\right)^2}{n}=\frac{4m^2}n.
    \end{aligned}
    $$
    Therefore, there exists an edge $v_1 v_2$ such that $d(v_1)+d(v_2)\geq \frac{\sum\limits_{ab\text{ is an edge}}(d(a)+d(b))}{e(G)}=\frac{4m}{n}.$

    Furthermore, if $G$ is $K_4$-free, $N(v_1)\cap N(v_2)$ is an independent set. Therefore, $|N(v_1)\cap N(v_2)|\geq |N(v_1)|+|N(v_2)|-|V(G)|$, which implies $\alpha(G)\geq \frac{4m}{n}-n.$
\end{proof}

By Lemma 4.2, $e(G)\geq 0.02959n^2$ implies $\alpha(G)\geq \frac{4m}{n}-n\geq 0.1836n$ when $G$ is $K_4$-free. 

\begin{lemma}
Let $G$ be a $K_4$-free graph on $n$ vertices, where $n>10^5$. If the independent number of $G$ satisfies $\alpha(G)<0.28n$ and $e(G)\geq 0.02959n^2$, then 
    $$
        D_{2,\infty}^b(G)< 0.074n^2.
    $$
\end{lemma} 

\begin{proof}
    When $e(G)>0.02959n^2$, Lemma 4.2 ensures that $\alpha(G)>0.18 n$ and there exists an edge $v_1 v_2$ such that $d(v_1)+d(v_2)\geq 1.18n$. We donate $N_1=N(v_1),N_2=N(v_2)$. Let $|N_1\setminus N_2|=a_1 n$, $|N_2\setminus N_1|=a_2 n$ and $|N_1\cap N_2|=c n$, then
    
    \begin{equation}
    \begin{cases}
         & a_1+a_2+2c\geq 1.18\\
         & a_1+a_2+c\leq 1   \\
         & 0.18< c < 0.28
    \end{cases}
    \end{equation}
    
    Without loss of generality, let $a_1\geq a_2$ which implies $a_1+c\geq \frac{1+0.18}{2}=0.59$ and $d(v_1)\geq d(v_2)$. In addition, since $N_1$ is triangle-free, for any $v\in N_1$, $N_{N_1}(v)$ is an independent set. Hence, $d_{N_1}(v)< 0.28n$. Similarly, $d_{N_2}(v)< 0.28n$.  
\\
\textbf{Case 3.1:}$a_2+c\leq 0.39$. 

This implies $a_1+c\geq 0.79$. Note that $e(N_1)= \frac 1 2\sum_{v\in N_1} d_{N_1}(v)\leq 0.14n\cdot |N_1|$. In addition, by Turán's theorem, the $K_4$-free graph $G- N_1$ contains at most $\frac1 3|V(G)\setminus N_1|^2$ edges. Therefore,

$$
\begin{aligned}
    e(G) & \leq e(N_1)+e(N_1,V(G)\setminus N_1))+e(G\setminus N_1) \\
    &\leq n^2(0.14(a_1+c)+(a_1+c)(1-a_1-c)+\frac 1 3 (1-a_1-c)^2)\\
    &=n^2\left(-\frac{2}{3}(a_1+c)^2+(\frac{1}{3}+0.14)(a_1+c)+\frac{1}{3}\right)\\
    &\leq n^2\left((-\frac{2}{3})\cdot0.79^2+(\frac{1}{3}+0.14)\cdot0.79+\frac{1}{3}\right) =0.2912 n^2.
\end{aligned}
$$
This implies $D_{2,\infty}^b(G)\leq 0.074n^2$ by Corollary 3.6. The last inequality holds because the derivative is below 0 when $d(v_1)\geq 0.79n$.\\
\textbf{Case 3.2:} $a_2+c\geq 0.45$. 

By (9), $a_1+a_2+c\geq 0.9$. In addition, $a_1+c\geq a_2+c\geq 0.45$, so we can choose a vertex subset $A$ of size $n/2$ such that $|A\cap N_1|=0.45n$ and $|A^c\cap N_1|=0.45n$. Therefore, $G[A]$ and $G[A^c]$ contains a triangle-free $0.45n-$vertex induced graph. By (5) of Lemma 3.1, $\max\{e(A),e(A^c) \}\leq (0.45n\cdot (0.5n-3/4\cdot 0.45n))<0.074 n^2.$
\\
\textbf{Case 3.3:} $0.39<a_2+c<0.45$. 

\textbf{Case 3.3.1:} $0.18<c\leq 0.23$. Here, $a_1\geq 1.18-a_2-2c> 0.5$.

Since $d_{N_1}(x)\leq 0.28n,d_{N_2}(x)\leq 0.28n$, $e(N_1\cap N_2, N_1\setminus N_2)=\sum_{v\in V(N_1\cup N_2)}d(v) \leq 0.28c n^2$. Since $n/2-|N_2|=(0.5-a_2-c)n\in(0.05n,0.11n)\subset(0,a_1]$, by Proposition 3.4 (let $B=N_1\setminus N_2$, $C=N_1\cap N_2$ and $p=n/2-|N_2|$), there exists a vertex subset $Y$ of size $(n/2-|N_2|)$ such that $N(v_2)\subset A\subset N(v_1)\cup N(v_2)$ and
$$e(Y,N_1\cap N_2)\leq \frac{\frac 1 2-a_2-c}{a_1} 0.28c n^2.$$

Let $A=Y\cup (N_2\setminus N_1)$. Since $N_1$ and $N_2$ are triangle-free, by Turán's Theorem, 
$$
\begin{aligned}
    e(A)\leq & e(N_1\cap N_2,Y)+e(N_2\setminus N_1,Y)+e(A\setminus N(v_2))+e(N(v_2))\\
    \leq & n^2\left(0.28c \frac{\frac 1 2-a_2-c}{a_1}+a_2(\frac 1 2 -a_2-c)+\frac{\left(\frac{1}{2}-a_2-c\right)^2}{4}+\frac{(a_2+c)^2}{4}\right)\\
    \overset{t=a_2+c}{=}&n^2\left( \frac{0.28}{a_1}c(\frac1 2-t)+(t-c)(\frac 1 2-t)+\frac{1}{4}\left(\frac 1 2-t\right)^2+\frac{1}{4}t^2 \right)   \\
    =&n^2\left(\left(\frac{0.28}{a_1}-1\right)c\left(\frac 1 2-t\right)-\frac 1 2 t^2+\frac 1 4 t+\frac 1 {16}\right)
\end{aligned}
$$

Since $1-c\geq a_1\geq 1.18-t-c$, $\left(\frac{0.28}{a_1}-1\right)c\geq \left(\frac{0.28}{1.18-t-c}-1\right)c \overset{def}{=}g_t(c)$.
When $t\in(0.39,0.45]$ is fixed, $\frac {\mathrm{d}g_t (c)} {\mathrm{d}c}=\frac{0.28(1.18-t)-(1.18-t-c)^2}{(1.18-t-c)^2}<0.$ Hence, the function reaches the maximum when $c=0.18$, which implies:
$$
\frac{e(A)} {n^2}\leq 0.18\left(\frac{0.28}{1-t}-1\right)\left(\frac 1 2-t\right)-\frac 1 2 t^2+\frac 1 4 t+\frac 1 {16}.
$$
Notice that $f(t):= 0.18\left(\frac{0.28}{1-t}-1\right)\left(\frac 1 2-t\right)-\frac 1 2 t^2+\frac 1 4 t+\frac 1 {16}$ is decreasing with respect to $t\in[0.39,0.45]$, because $f^{'}(t)=0.43-t-\frac{0.0252}{(t-1)^2}<0$ when $t\in[0.39,0.45]$.
Hence, $e(A)\leq n^2 f(0.39)<0.0733n^2$. On the other hand, $|A^c\cap N_1|\geq (a_1+a_2+c-0.5)n\geq (0.68-c)n\geq 0.45n$, so by (5) of Lemma 3.1, $e(A^c)< 0.074n^2$. Therefore, $D_{2,\infty}^b(G)\leq 0.074n^2$.

\textbf{Case 3.3.2:} $0.23<c\leq 0.28$. To be precise, we are considering the following five cases: $c\in(0.22 + 0.01i, 0.23 + 0.01i]$ for $i \in [5]$.

 Since $a_2+c\geq 0.45$ and  $a_1+a_2+c\geq 1.18-c$, $0.5-0.01i-a_2-c\in[0,a_1]$. Hence, similarly to \textbf{Case 3.3.1}, by Proposition 3.4 there exists a vertex subset $Y_1\subset (N_1\setminus N_2)$ with $|Y_1|=(0.5-0.01i-a_2-c)n$ such that 
 $$e(Y_1,N_1\cap N_2)\leq\frac{0.5-0.01i-a_2-c}{a_1}e(N_1\setminus N_2,N_1\cap N_2)\leq n^2\frac{0.5-0.01i-a_2-c}{a_1}0.28cn^2.$$
 
 Now, we can choose a vertex subset $A$ of size $n/2$ such that $A\supset (N_2\cup Y_1)$ and $|A^c\cap (N_1\setminus N_2)|\geq 0.45$. This holds because, for each $i$, $a_1+a_2+c\geq 1.18-c\geq 0.95-0.01i$ when $c\in (0.22+0.01i,0.23+0.01i]$. Here, the condition $|A^c\cap(N_1\setminus N_2)|\geq 0.45n$ ensures that $e(A^c)<0.074 n^2$ by (5) of Lemma 3.1. On the other hand,
 $$
 \begin{aligned}
     e(A)=& e(Y_1,N_1\cap N_2)+e(Y_1,N_2\setminus N_1)+e(Y_1)+e(N_2) \\
     &+e(Y_1\cup N_2,A\setminus (Y_1\cup N_2))+e(A\setminus (Y_1\cup N_2))
 \end{aligned}
 $$

Let $t=a_2+c$. Now, $\frac{e(Y_1,N_1\cap N_2)}{n^2}\leq 0.28c\cdot \frac{0.5-0.01i-a_2-c}{a_1} \leq 0.28c\cdot\frac{0.5-0.01i-t}{1.18-t-c}$. Moreover, $G(I_1)$ is triangle-free, so $e(I_1)\leq \frac{1}{4}(0.5-0.01i-t)^2$; while $G(N_2)$ is triangle-free graph and contains an independent set of size $c>\frac{t}{2}$, so $e(I_2)\leq (t-c)c$ by Lemma 3.2. Therefore, we let

$$
\begin{aligned}
    \frac{e(A)}{n^2} &\leq 0.28c\frac{0.5-0.01i-t}{1.18-t-c}+(t-c)(0.5-0.01i-t)+(t-c)c+\frac{1}{4}(0.5-0.01i-t)^2
    \\&+(0.5-0.01i)\cdot 0.01i+\frac{1}{3}(0.01i)^2\overset{def}{=}F_i(c,t)
\end{aligned}
$$
where $0.39<t<0.45$, $c\in(0.22+0.01i,0.23+0.01i]$ and $i\in [5]$. Hence,

$$
\begin{aligned}
\frac{\mathrm{d}F_i(c,t)}{\mathrm{d}c}
&=(0.5-0.01i-t)\left(\frac{0.28}{1.18-t-c}+\frac{0.28c}{(1.18-t-c)^2}-1\right)+t-2c\\
&<(0.5-0.01i-t)\left(\frac{0.28}{0.5-0.01i}+\frac{0.28(0.23+0.01i)}{(0.5-0.01i)^2}-1\right)+0.45-0.44-0.02i\\
&\leq (0.5-0.01i-t)\left(\frac{0.28}{0.45}+\frac{0.28\cdot0.28}{0.45^2}-1\right)-0.01\\
&\leq 0.1\cdot 0.01-0.01<0
\end{aligned}
$$

Therefore, $F_i(c,t)\leq F_i(0.22+i,t)$. Actually,

\begin{align}
F_i(0.22+i,t)=
\begin{cases}
-\frac{3t^2 }{4}+\frac{141t}{200}-\frac{12077}{120000}+ \frac{\frac{161t}{2500}-\frac{7889}{250000}}{t-\frac{19}{20}} 
& i=1\\
-\frac{3t^2 }{4}+\frac{18t}{25}-\frac{791}{7500}+\frac{\frac{42t}{625}-\frac{504}{15625}}{t-\frac{47}{50}} 
& i=2\\
-\frac{3t^2 }{4}+\frac{147t}{200}-\frac{883}{8000}+\frac{\frac{7t}{100}-\frac{329}{10000}}{t-\frac{93}{100}}
& i=3\\
-\frac{3t^2 }{4}+\frac{3t}{4}-\frac{3461}{30000}+\frac{\frac{91t}{1250}-\frac{2093}{62500}}{t-\frac{23}{25}}
& i=4\\
-\frac{3t^2 }{4}+\frac{153t}{200}-\frac{14453}{120000}+\frac{\frac{189t}{2500}-\frac{1701}{50000}}{t-\frac{91}{100}}
& i=5
\end{cases}    
\end{align}

Below, we will show that $G_i(t):=F_i(0.22+i,t)(0.96-0.01i-t)-0.0739(0.96-0.01i-t)<0$ which implies $e(A) <0.074 n^2$. 
\begin{align}
    G_i(t) &= 
\begin{cases} 
\left(-\frac{3}{4}t^2 + \frac{141}{200}t - \frac{12077}{120000}\right) \left(\frac{19}{20} - t\right) - \frac{161}{2500}t + \frac{7889}{250000} - 0.0739 \left(t - \frac{19}{20}\right) & i = 1 \\
\left(-\frac{3}{4}t^2 + \frac{18}{25}t - \frac{791}{7500}\right) \left(\frac
 {47}{50} - t\right) - \frac{42}{625}t + \frac{504}{15625} - 0.0739 \left(t - \frac{47}{50}\right) & i = 2 \\
\left(-\frac{3}{4}t^2 + \frac{147}{200}t - \frac{883}{8000}\right) \left(\frac{93}{100} - t\right) - \frac{7}{100}t + \frac{329}{10000} - 0.0739 \left(t - \frac{93}{100}\right) & i = 3 \\
\left(-\frac{3}{4}t^2 + \frac{3}{4}t - \frac{3461}{30000}\right) \left(\frac{23}{25} - t\right) - \frac{91}{1250}t + \frac{2093}{62500} - 0.0739 \left(t - \frac{23}{25}\right) & i = 4 \\
\left(-\frac{3}{4}t^2 + \frac{153}{200}t - \frac{14453}{120000}\right) \left(\frac{91}{100} - t\right) - \frac{189}{2500}t + \frac{1701}{50000} - 0.0739 \left(t - \frac{91}{100}\right) & i = 5 
\end{cases}\\
&< \begin{cases}
0.75 t^3 - 1.4175 t^2 + 0.77990 t - 0.13416 & i = 1 \\
0.75 t^3 - 1.425 t^2 + 0.78898 t - 0.13625 & i = 2 \\
0.75 t^3 - 1.4325 t^2 + 0.79783 t - 0.13838 & i = 3 \\
0.75 t^3 - 1.44 t^2 + 0.80648 t - 0.14054 & i = 4 \\
0.75 t^3 - 1.4475 t^2 + 0.81490 t - 0.14273 & i = 5
\end{cases}
\end{align}

The figure of functions in (12) is given below.

\begin{figure}[H]
    \centering
    \includegraphics[width=0.5\linewidth]{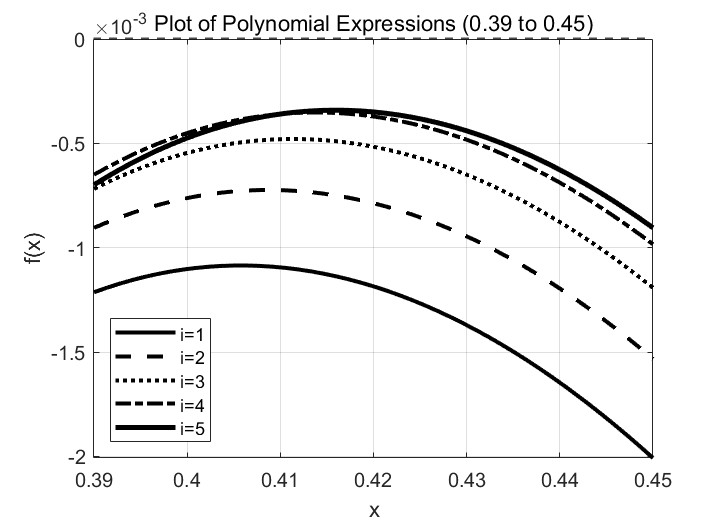}
    \caption{Graph of (12)}
    \label{fig:enter-label}
\end{figure}

Therefore, the $A$ we choose ensures $\max\{e(A),e(A^c)\}\leq 0.074n^2$.
\end{proof}

\begin{proof}[Proof of Theorem 1.8]
    If $e(G)\leq0.02959n^2$, then by Corollary 3.6, $D_{2,\infty}^b(G)<0.074n^2$. If $e(G)>0.02959n^2$, by Lemmas 4.1 and 4.3, $D_{2,\infty}^b(G)<0.074n^2$.
\end{proof}

\section{An upper bound for $I\vee H$}

\begin{proof}[Proof of Theorem 1.9]

Let  $G= H\vee I$, where $I$ is an independent set and $H$ is a triangle-free graph $H$. We will prove the Theorem 1.9 by considering the four cases regrading the order of the independents set.  

\textbf{Case 1}: $|I|>0.62n$. Now, there exists 2 disjoint vertex-subsets $I_1$ and $I_2$ with $|I_1|=|I_2|=0.31n$. Choose a vertex subset $A$ of size $n/2$ such that $A\supset I_1$ and $A^c\supset I_2$. Since $|I_1|=|I_2|\geq\frac{1}{3}|A|=\frac{1}{3}|A^c|$, by (7) of Lemma 3.3

$$
\max\{e(A),e(A^c)\}\leq \frac{1}{4}(0.5n-0.31n)(0.5n+3\cdot 0.31n)=0.067925n^2<\frac 5 {72} n^2
$$

\textbf{Case 2} $|I|\in(\frac{7}{12}n,0.62n]$. 
If $e(H)\leq 0.16|H|^2$, then $e(G)<0.271n^2$, which implies $D_{2,\infty}^b(G)<\frac{5}{72}n^2$ by Corollary 2.6.

If $e(H)> 0.16|H|^2$, we can delete $\frac{|H|^2}{25}$ edges to make $H$ bipartite by Theorem 3.7. This implies $G$ can be made tripartite by deleting $\frac{|H|^2}{25}$ edges. Moreover, $|H|=n-|I|<\frac {5n}{12}$. Therefore, by Theorem 1.4, 
$$D_{2,\infty}(G)\leq \frac{|H|^2}{25}+\frac{n^2}{16}< \frac 5 {72} n^2.$$

\textbf{Case 3} $|I|\in [0.5n,\frac{7}{12}n]$. Let $|I|=\alpha n$. Here, $\alpha \in[0.5,7/12]$.  If $\frac{e(H)}{n^2}< \frac{5}{72}-(1-\alpha) (\alpha-\frac 1 2)$, then choose a vertex subset $A$ of size $\frac{n}{2}$ such that $A^c\supset V(H)$. Then, $\max\{e(A),e(A^c)\}<\frac{5}{72}n^2$. 

If $\frac{e(H)}{n^2}\geq \frac{5}{72}-(1-\alpha) (\alpha-\frac 1 2)$, then by Theorem 3.8, we can delete $\left(e(H)-\frac{4 e(H)^2}{|H|^2}\right)$ edges to make $H$ bipartite. This implies $G$ can be made tripartite by deleting $e(H)-\frac{4 e(H)^2}{|H|^2}$ edges. Let $e_0=e(H)/n^2$. By Theorem 1.4, 
$$
\begin{aligned}
\frac{D_{2,\infty}^b(G)}{n^2} \leq& \frac{1}{16}+\frac{e(H)}{n^2}-\frac{4 e(H)^2}{|H|^2n^2}\\
=& \frac{1}{16}+e_0-\frac{4e_0^2}{(1-\alpha)^2}\overset{def}{=}g_\alpha(e_0)
\end{aligned}.
$$
where $\alpha \in[0.5,7/12]$ and $e_0\geq \frac{5}{72}-(1-\alpha) (\alpha-\frac 1 2)$. Actually, for a fixed $\alpha$, the function $g_\alpha(e_0)$ is decreasing with respect to $[\frac{(1-\alpha)^2}{8},+\infty)$. Moreover, when $\alpha \in[0.5,7/12]$, $\frac{5}{72}-(1-\alpha) (\alpha-\frac 1 2)-\frac{(1-\alpha)^2}{8}=\frac{1}{72}(21\alpha -16)(3\alpha -2)>0$. Hence, $g_\alpha(e_0)\leq g_\alpha(\frac{5}{72}-(1-\alpha) (\alpha-\frac 1 2))$. Let $f(\alpha):=g_\alpha(\frac{5}{72}-(1-\alpha) (\alpha-\frac 1 2))-\frac{1}{16}=\frac{5}{72}+(\alpha-1)(\alpha-\frac 1 2)-\frac{25}{1296(1-\alpha)^2}-\frac{5}{9}\cdot\frac{\frac{1}{2}-\alpha}{1-\alpha}-4(\alpha-\frac 1 2)^2$.
Note that when $\alpha \in[0.5,7/12]$, 

\begin{align}
    f^{'}(\alpha)&=\frac{5}{18}\cdot \frac 1 {(\alpha-1)^2}\left(1-\frac{5}{36(1-\alpha)}\right)-6\alpha +\frac 5 2\\
    &\leq \frac{5}{18}\cdot \frac 1 {(\alpha-1)^2}\cdot \frac{2}{3}-6\alpha +\frac{5}{2}\\
    &=\frac{1}{(\alpha-1)^2}\left(\frac{5}{27}-(6\alpha -\frac 5 2)(\alpha-1)^2\right) <0
\end{align}
(15) holds because the function $f_0(\alpha)=(6\alpha -2.5)(\alpha-1)^2$ is increasing with respect to $(-\infty,11/18]$, which implies $(6\alpha -\frac 5 2)(\alpha-1)^2-\frac {5}{27}\leq (6 \cdot 7/12-5/2)(7/12-1)^2<0$. Hence, if $\alpha \in[0.5,7/12]$, $f(\alpha)\leq f(7/12)=1/144$.  

Therefore, $\frac{D_{2,\infty}^b(G)}{n^2}\leq \frac{1}{16}+f(\frac{7}{12})\leq 5/72$. 

\textbf{Case 4} $\alpha(G)<0.5n$. If $e(H)\leq (\frac{5}{18}-0.001)n^2-|H|(n-|H|)$ (which implies $e(G)\leq (\frac{5}{18}-0.001)n^2$), then $\frac{D_{2,\infty}^b(G)}{|G|^2}<\frac{5}{72}$ by Corollary 3.6. 

If $e(H)> (\frac{5}{18}-0.001)n^2-|H|(n-|H|)$, then there exists a vertex $v\in V(H)$ satisfying $d_{H}(v)\geq \frac{2e(H)}{|H|}\geq \frac{0.553n^2}{|H|}+2|H|-2n$. Since $H$ is triangle-free, $N_H(v)$ is an independent set with $|N_H(v)|>{0.553n^2}{|H|}+2|H|-2n$. Hence, we can let $I_0$ be the independent set of $H$ with $|I_0|=\frac{0.553n^2}{|H|}+2|H|-2n$. By AM-GM inequality, $|V(I\cup I_0)|\geq \frac{0.553n^2}{|H|}+|H|-n\geq (2\sqrt{0.553}-1)n^2\geq 0.486 n$.

If $|V(I\cup I_0)|>\frac{n}{2}$, we can choose a vertex subset $A$ of size $\frac n 2$ such that $I\subset A\subset I\cup I_0$. Now, $A$ is triangle-free, and $A^c$ which is a subgraph of $H$ is also triangle-free. Therefore, by Mantel's theorem, $\frac{D_{2,\infty}^b(G)}{n^2}\leq \frac 1 {16} n^2<\frac{5}{72} n^2$.

If $|V(I\cup I_0)|<\frac{n}{2}$, we can choose a vertex subset $A$ of size $\frac n 2$ such that $A\supset I\cup I_0$. Note that $I\cup I_0$ is a triangle-free graph with $|I\cup I_0|\geq 0.486 n$, so by (2) of Lemma 3.1 $e(A)\leq |I\cup I_0|(|A|-\frac{3}{4}|I\cup I_0|)=0.065853n^2<\frac{5}{72}n^2$. On the other hand, since $A^c$ is triangle-free, $e(A)\leq \frac{(n/2)^2}{4}<\frac{5}{72}n^2$. Therefore, $D_{2,\infty}^b(G)<\frac{5}{72}n^2$ in this case.
    
\end{proof}
\section*{Acknowledgement}
This work is partly supported by the National Natural Science Foundation of China (Nos.12371354, 
12161141003) and Science and Technology Commission of Shanghai Municipality (No. 22JC1403600),
National Key R\&D Program of China under Grant No. 2022YFA1006400.

\bibliographystyle{abbrv}
\bibliography{ref}

\appendix
\end{document}